\documentclass[reqno,12pt]{amsart} 

\usepackage{hyperref}
\usepackage{comment}
\usepackage{tikz}

\newtheorem{theorem}{Theorem}
\newtheorem{proposition}{Proposition}

\newenvironment{theoremp}[1]
  {\renewcommand{\thetheorem}{\ref{#1}$'$}%
   \addtocounter{theorem}{-1}%
   \begin{theorem}}
   {\end{theorem}}
 
\theoremstyle{remark}
\newtheorem{remark}{Remark}

\numberwithin{equation}{section}

\allowdisplaybreaks

\newcommand{\Ga}{\Gamma}
\newcommand{\la}{\lambda}

\newcommand{\N}{\mathbb N}
\newcommand{\Z}{\mathbb Z}
\newcommand{\R}{\mathbb R}
\newcommand{\C}{\mathbb C}

\author{Michael J.\ Schlosser}
\address{Fakult\"at f\"ur Mathematik, Universit\"at Wien,
Oskar-Morgenstern-Platz 1, A-1090 Vienna, Austria}
\email{michael.schlosser@univie.ac.at}
\thanks{The author's research was partly supported by
FWF Austrian Science Fund grant
\href{https://doi.org/10.55776/P32305}{10.55776/P32305}.}

\title[Orthogonal functions and a multilateral matrix inverse]{A
  family of orthogonal functions on the
unit circle and a new multilateral matrix inverse}

\dedicatory{Dedicated to Tom H.\ Koornwinder on the occasion
of his 80th birthday}



\begin{document}

\begin{abstract}
Using Bailey's very-well-poised $_6\psi_6$ summation, we
show that a specific sequence of well-poised bilateral basic
hypergeometric $_3\psi_3$ series form a family of
orthogonal functions on the unit circle.
We further extract a bilateral matrix inverse from
Dougall's ${}_2H_2$ summation which we use,
in combination with the Pfaff--Saalsch\"utz summation,
to derive a summation for a particular bilateral
hypergeometric $_3H_3$ series.
We finally provide multivariate extensions of the bilateral
matrix inverse and the $_3H_3$ summation in the setting of
hypergeometric series associated to the root system $A_r$.
\end{abstract}

\maketitle

\section{Introduction}\label{sec:intro}
Orthogonality of functions and matrix inversion are two topics
that are closely tied together (see e.g.\ \cite{CK}).
In this paper, we make use of summation theorems for bilateral
(and multilateral) hypergeometric and basic hypergeometric
series to derive results belonging to these two topics:

Firstly, we establish the orthogonality of a family of functions on the unit
circle that can be represented in terms of well-poised bilateral basic
hypergeometric $_3\psi_3$ series.
We achieve this by an application of Bailey's very-well-poised
$_6\psi_6$ summation theorem. (The reader is kindly referred to \cite{Sl}
and to \cite{GR} regarding standard notions and results in the theories of
hypergeometric and basic hypergeometric series, respectively.)

Secondly, we utilize Dougall's ${}_2H_2$ summation
to find a new bilateral matrix inverse (with explicit entries
consisting of products of gamma functions) which we use to derive,
in combination with an instance of the Pfaff--Saalsch\"utz summation,
by application of \textit{inverse relations} (cf.\ \cite{R})
a summation for a particular bilateral hypergeometric $_3H_3$ series.
We then provide multivariate extensions of these results
in the setting of hypergeometric series associated to the root
system $A_r$. This is achieved by utilizing an $A_r$ $_2H_2$
summation theorem by Gustafson, from which we extract a multilateral
matrix inverse (again, with explicit entries
consisting of products of gamma functions). We use our inversion result,
in combination with an $A_r$ Pfaff--Saalsch\"utz summation by Milne,
to derive a summation for a particular multilateral
hypergeometric $_3H_3$ series associated to the root system $A_r$.

We would like to point out that the use of multivariate matrix inverses
in special functions (in particular, in the theory of ordinary, basic,
and elliptic hypergeometric series associated with root systems)
has proved to be a very powerful tool in obtaining new results.
Some papers that demonstrate this are
\cite{B,BM,BS,C,CG,DJ,KS,LS1,LS2,LM,M2,ML,R,Ro,RS,S1,S2,S5,S6,S7,SW,ZW}.

Our paper is organized as follows. In Section~\ref{sec:sb}
we review some notions of ordinary hypergeometric and basic
hypergeometric series and list some of the classical results (summations)
that we employ. In Section~\ref{sec:of} we are concerned with a
particular family of basic hypergeometric functions (that can be represented
as multiples of $_3\psi_3$ series) for which we establish their orthogonality
on the unit circle. Section~\ref{sec:mi} features bilateral and
multilateral matrix inverses together with univariate and multivariate
$_3H_3$ summations as applications. Finally, in Section~\ref{sec:con}
we provide a conclusion where an open problem is stated for future
research.

\section{Some background on unilateral and bilateral
hypergeometric and basic hypergeometric series}\label{sec:sb}
Let $\Z$ and $\N_0$ denote the sets of integers and of non-negative integers,
respectively. We shall generally assume the parameters $z,a,b,c,\ldots$
that appear in our expressions to be complex numbers,
and further assume the special variable $q$ (called the ``base'')
to satisfy $0<|q|<1$ unless explicitly stated otherwise.
(In particular, when considering orthogonal functions in
Section~\ref{sec:of} we shall restrict $q$ to be real
with $0<q<1$.)

For $k\in\Z$
the \textit{shifted factorial} is defined as

\[
(a)_k:=\frac{\Gamma(a+k)}{\Gamma(a)}.
\]

Similarly, for $k\in\Z$
the \textit{$q$-shifted factorial} is defined as

\[
(a;q)_k:=\frac{(a;q)_\infty}{(aq^k;q)_\infty}\qquad\text{where}\quad
(a;q)_\infty=\prod_{j\ge 0}(1-a q^j).
\]

We find it convenient to use the following compact product notations:
\begin{alignat*}{2}
\Ga(a_1,\dots,a_m):={}&\Ga(a_1)\dotsm\Ga(a_m),&&\\
(a_1,\dots,a_m)_k:={}&(a_1)_k\dotsm(a_m)_k, &&k\in\Z,\\
(a_1,\dots,a_m;q)_k:={}&(a_1;q)_k\dotsm(a_m;q)_k, \qquad&&
k\in\Z\cup\{\infty\}.
\end{alignat*}
We recall the following definitions of unilateral and bilateral 
(basic) hypergeometric series (cf.\ \cite{Sl} and \cite{GR}).
In the following, $a_0,a_1,\ldots,a_r$ are ``upper''
(complex) parameters, $b_1,\ldots,b_r$ are ``lower''
(complex) parameters, and $z$ is a power series variable,
in addition to the base $q$ appearing in the basic variants of
the respective series.

The (\textit{unilateral}) hypergeometric ${}_{r+1}F_r$ series
is defined by
\begin{equation}
{}_{r+1}F_r\!\left[\begin{matrix}a_0,a_1,\dots,a_r\\
b_1,b_2,\dots,b_r\end{matrix};z\right]
:=\sum_{k=0}^{\infty}\frac{(a_0,a_1,\dots,a_r)_k}
{(1,b_1,\dots,b_r)_k}z^k.
\end{equation}
The $_{r+1}F_r$ series  converges absolutely (if it is not a finite sum)
for $|z|<1$ or for $|z|=1$ when
$\Re(b_1+\cdots+b_r-a_0-\ldots-a_r) >0$,
and converges conditionally if $|z|=1$, $z\neq 1$, when
$0\ge \Re(b_1+\cdots+b_r-a_0-\ldots-a_r)>-1$.

The \textit{bilateral} hypergeometric ${}_rH_r$ series
is defined by
\begin{equation*}
{}_rH_r\!\left[\begin{matrix}a_1,a_2\dots,a_r\\
b_1,b_2,\dots,b_r\end{matrix};z\right]
:=\sum_{k=-\infty}^{\infty}\frac{(a_1,\dots,a_r)_k}
{(b_1,\dots,b_r)_k}z^k.
\end{equation*}
The $_rH_r$ series converges absolutely (if it is not a finite sum)
for $|z|=1$ when $\Re(b_1+\cdots+b_r-a_1-\ldots-a_r) >-1$,
and converges conditionally if $|z|=1$, $z\neq 1$, when
$1\ge \Re(b_1+\cdots+b_r-a_1-\ldots-a_r)>0$.

The (\textit{unilateral}) basic hypergeometric
${}_{r+1}\phi_r$ series is defined by
\begin{equation*}
{}_{r+1}\phi_r\!\left[\begin{matrix}a_0,a_1,\dots,a_r\\
b_1,b_2,\dots,b_r\end{matrix};q,z\right]
:=\sum_{k=0}^{\infty}\frac{(a_0,a_1,\dots,a_r;q)_k}
{(q,b_1,\dots,b_r;q)_k}z^k.
\end{equation*}
The ${}_{r+1}\phi_r$ series converges absolutely
(if it is not a finite sum) for $|z|<1$.

The \textit{bilateral} basic hypergeometric
${}_r\psi_r$ series is defined by
\begin{equation*}
{}_r\psi_r\!\left[\begin{matrix}a_1,a_2\dots,a_r\\
b_1,b_2,\dots,b_r\end{matrix};q,z\right]
:=\sum_{k=-\infty}^{\infty}\frac{(a_1,\dots,a_r;q)_k}
{(b_1,\dots,b_r;q)_k}z^k.
\end{equation*}
The ${}_r\psi_r$ series converges absolutely (if it is not a finite sum)
for $|b_1\cdots b_r/a_1\cdots a_r|<|z|<1$.

We now note a number of well-known classical summations
that we encounter in this paper, for easy reference.

The \textit{Pfaff--Saalsch\"utz summation} \cite[Equation~(2.3.1.3)]{Sl}
is given by
\begin{equation}\label{eq:3F2}
{}_3F_2\!\left[\begin{matrix}a,\,b,\,-m\\
c,\,a+b+1-m-c\end{matrix}\,;1\right]=\frac{(c-a,c-b)_m}{(c,c-a-b)_m},
\end{equation}
where $m$ is a non-negative integer.
As the series in \eqref{eq:3F2} terminates (due to the appearance
of $-m$ as an upper parameter),
no condition of convergence is needed.

\textit{Dougall's ${}_2H_2$ summation} \cite[Equation~(6.1.2.1)]{Sl} is
\begin{equation}\label{eq:2H2}
{}_2H_2\!\left[\begin{matrix}a,\,b\\
c,\, d\end{matrix}\,;1\right]=\frac{\Ga(1-a,1-b,c,d,c+d-a-b-1)}
{\Ga(c-a,c-b,d-a,d-b)},
\end{equation}
where $\Re(c+d-a-b-1)>0$.

In \cite[Section~1.6]{K}, Koornwinder noted the following ${}_1H_1$ summation,
as a (formal) limiting case of Ramanujan's $_1\psi_1$ summation formula:
\begin{equation}\label{eq:1H1}
{}_1H_1\!\left[\begin{matrix}a\\
c\end{matrix}\,;z\right]=\frac{\Ga(1-a,c)}
{\Ga(c-a)}\frac{(1-z)^{c-a-1}}{(-z)^{c-1}},
\end{equation}
where $|z|=1$, $z\ne 1$, and $\Re(c-a-1)>0$.
We note that the ${}_1H_1$ summation also results as a formal limiting case
of Dougall's ${}_2H_2$ summation. (Replace $b$ by $dz$
in \eqref{eq:2H2} and let $d\to\infty$.)

In our present investigations involving basic hypergeometric series, we
deal with series that are \textit{very-well-poised}
(cf.\ \cite{GR} for the terminology; other standard terminologies we
adopt, is that of a series being \textit{well-poised}, being
\textit{balanced} and, more generally, \textit{$k$-balanced}).
The following two summations in \eqref{eq:8phi7} and \eqref{eq:6psi6}
concern very fundamental ones for basic hypergeometric series
(as they stand on the very top of the hierarchy of summation theorems
for unilateral and bilateral basic hypergeometric series):
\textit{Jackson's ${}_8\phi_7$ summation}~\cite[Equation~(2.6.2)]{GR}
is given by
\begin{align}\label{eq:8phi7}
{}_8\phi_7\!\left[\begin{matrix}a,\;qa^{\frac12},\,-qa^{\frac12},\,
b,\,c,\,d,\,e,\,q^{-m}\\
a^{\frac12},\,-a^{\frac12},\,aq/b,\,aq/c,\,aq/d,\,aq/e,\,aq^{1+m}\end{matrix}
  \,;q,q\right]&\notag\\
  =\frac{(aq,aq/bc,aq/bd,aq/cd;q)_m}
{(aq/b,aq/c,aq/d,aq/bcd;q)_m}&,
\end{align}
where $a^2q=bcdeq^{-m}$, $m$ being a non-negative integer.
As the series in \eqref{eq:8phi7} terminates (due to the appearance
of $q^{-m}$ as an upper parameter),
no condition of convergence is needed.

\textit{Bailey's ${}_6\psi_6$ summation} \cite[Equation~(5.3.1)]{GR} is
\begin{align}\label{eq:6psi6}
&{}_6\psi_6\!\left[\begin{matrix}qa^{\frac12},\,-qa^{\frac12},\,
b,\,c,\,d,\,e\\
a^{\frac12},\,-a^{\frac12},\,aq/b,\,aq/c,\,aq/d,\,aq/e\end{matrix}
  \,;q,\frac{a^2q}{bcde}\right]\notag\\
&=\frac{(q,aq,q/a,aq/bc,aq/bd,aq/be,aq/cd,aq/ce,aq/de;q)_\infty}
{(q/b,q/c,q/d,q/e,aq/b,aq/c,aq/d,aq/e,a^2q/bcde;q)_\infty},
\end{align}
where $|a^2q/bcde|<1$.

A combinatorial proof of \eqref{eq:8phi7}, in the more general
setting of elliptic hypergeometric series (cf.\ \cite[Chapter~11]{GR}),
is given in \cite[Equations~(2.18)--(2.20)]{S4}, whereas a purely
combinatorial proof of \eqref{eq:6psi6} still appears to be missing
as of today. For some applications of Bailey's ${}_6\psi_6$ summation
in \eqref{eq:6psi6} to number theory, see \cite[Section~3]{A} (where
also an elementary analytic proof of \eqref{eq:6psi6} can be found that
utilizes $q$-difference equations together with the uniqueness of a
Laurent series about the origin).
Further, we would like to stress that the $_2H_2$ summation
in \eqref{eq:2H2} is \textit{not} a direct special case of
the $_6\psi_6$ summation in \eqref{eq:6psi6}.
For a discussion of how to obtain \eqref{eq:2H2}
in several steps from \eqref{eq:6psi6} by a careful application
of Tannery's theorem, see \cite{BT}.

Finally, we would like to fix the notation
$|n|:=n_1+\cdots+n_r$ (and similarly for $|k|$ and $|l|$;
not to be confused with a norm)
for multi-indices of integers when dealing with
multivariate hypergeometric and basic hypergeometric series.
In the multivariate setting, we work here in the theory of hypergeometric
and basic hypergeometric series associated with the root system $A_r$.
(See \cite{S8} for a survey of such series.)

\section[Orthogonal functions on the unit circle]{Orthogonal functions
on the unit circle}\label{sec:of}

Let $f$ and $g$ be two complex-valued functions on the
\textit{unit circle} $|z|=1$. 
We consider the \textit{inner product}
\begin{equation*}
\langle f,g\rangle:=\frac 1{2\pi\mathrm i}\int_{|z|=1}
f(z)\overline{g(z)}\,w(z)\frac{{\mathrm d}z}z,
\end{equation*}
where $w(z)$ is a positive definite weight function on the unit circle.

The two functions $f$ and $g$ are \textit{orthogonal}
with respect to the weight $w$ if
\begin{equation*}
\langle f,g\rangle=0.
\end{equation*}
The sequence of functions $(f_n)_n$ on the unit circle
forms a \textit{family of orthogonal functions} if
\begin{equation*}
\langle f_n,f_m\rangle=h_n\delta_{n,m},
\end{equation*}
for some sequence of positive real numbers $(h_n)_n$
(the ``squared norms'').

The simplest example for a family of orthogonal functions
on the unit circle is $(f_n)_{n\in\Z}$ with
\begin{equation*}
f_n=z^n,\qquad\text{with weight function}\quad w(z)=1.
\end{equation*}

Indeed, since $\overline{z}=z^{-1}$ on the unit circle it is easy to compute
\begin{align}\label{eq:Fo}
\langle z^n,z^m\rangle&=\frac 1{2\pi\mathrm i}\int_{|z|=1}
z^{n-m}\frac{{\mathrm d}z}z\notag\\&=\frac 1{2\pi}\int_{0}^{2\pi}
e^{\mathrm i\theta(n-m)}{\mathrm d}\theta=
\begin{cases}1&\;\;\text{if $n=m$},\\
\frac{e^{i\theta(n-m)}}{2(n-m)\pi\mathrm i}\Big|_{\theta=0}^{2\pi}=0
&\;\;\text{if $n\ne m$},
\end{cases}
\end{align}
which is commonly referred to as \textit{Fourier orthogonality}.

Many generalizations of this simple family of orthogonal functions exist.
Usually different weight functions $w(z)$ are considered
with corresponding families of functions that are orthogonal
with respect to the given weight function.
Moreover, in many of the classic cases the functions are actually
\textit{polynomials} -- having the advantage that Favard's theorem
then holds. (See
\cite{ENZG} for the unit circle analogue of Favard's theorem.)
Most notable extensions of the simple family of orthogonal functions
$(z^n)_{n}$ are the Rogers--Szeg\H{o} polynomials, the $q$-Hermite
polynomials, and the Askey--Wilson polynomials.
See \cite{I} for a treatise on
various important families of orthogonal and $q$-orthogonal
polynomials, their properties, with motivation and proofs.

For the following theorem, we assume
$0<q<1$ (and fix the branch of square root of $q$),
$z,g,b,c\in\C\setminus\{0\}$ with $c=\overline{b}$,
and choose $\Re(g)=0$ (thus $\overline{g}=-g$ and $g^2<0$).
We further tacitly assume the parameters to be such that
no poles or zeroes appear in \eqref{eq:fn} and in \eqref{eq:hn1}.
(We will make similar assumptions on the parameters
in other results in this paper.)
\begin{theorem}[A family of orthogonal functions]\label{thm:orthf}
  Let
  \begin{subequations}\label{eq:fn}
\begin{align}\label{eq:fna}
    F_n=F_n(z;g,b,c\,|q):={}&\sum_{k\in\Z}F_{n,k}(g,b,c\,|q)z^k,\\
\intertext{where}
F_{n,k}(g,b,c\,|q):={}&\frac{(1-gq^k)}{(1-g)}
\frac{(b;q)_{n+k}(g^2/c;q)_{k-n}}{(cq;q)_{n+k}(g^2q/b;q)_{k-n}}q^{\frac k2}.
\end{align}                    
\end{subequations}
Then $\big(F_n\big)_{n\in\Z}$ is an \textit{infinite family of
  orthogonal functions} on the unit circle with respect to the
weight function $w(z)=1$
and squared norm evaluation
\begin{align}\label{eq:hn1}
  &h_n=h_n(g,b,c\,|q)\notag\\
  &:=
\frac{(q,q,g^2q,q/g^2,g^2q/bc,bcq/g^2,cq/b,bq/c;q)_\infty}
{(g^2q/b,bq/g^2,g^2q/c,cq/g^2,bq,q/b,cq,q/c;q)_\infty}
  \frac{(1-bc/g^2)}{(1-bcq^{2n}/g^2)}q^n.
\end{align}
\end{theorem}
\begin{remark}
Since $q,g^2\in\R$ and $c=\overline{b}$, it is clear
that $h_n\in\R$ for all $n\in\N_0$. Since $q>0$ and $g^2<0$,
it is further easy to see that $h_n>0$. Indeed, $h_n$ can be written as
\begin{align}\notag
  h_n&=\left| \frac{(q,bq/\overline{b};q)_\infty}
       {(g^2q/b,bq/g^2,bq,q/b;q)_\infty}\right|^2\\
&\quad\;\times (g^2q,q/g^2,g^2q/|b|^2,|b|^2q/g^2;q)_\infty
\frac{(1-|b|^2/g^2)}{(1-|b|^2q^{2n}/g^2)}q^n,
\end{align}
where each individual factor in the last line is positive,
hereby furnishing the claimed positivity of $h_n$.
Further, the series in \eqref{eq:fn} converges absolutely.
The two necessary conditions of convergence are
$|q^{\frac 12}|<1$ and $|c^2q^{\frac 12}/b^2|<1$.
They are trivially satisfied as they are already subsumed
by the condition $0<q<1$ and the fact that $|b|=|c|$.

What is interesting about Theorem~\ref{thm:orthf}, is that
the positive definite weight function $w(z)$ is trivial
(we have $w(z)=1$) which makes it surprising (to us)
that the orthogonality of these functions has not been
discovered earlier. In the special limiting case $0>c=b\to g^2$,
the above functions $F_n$ reduce  to
$$
\lim_{b\to g^2}F_n(z;g,b,b\,|q)=\frac{1+g}{1+gq^n}q^{\frac n2}z^n
$$
(as the sum over $k$ in \eqref{eq:fna} reduces to just one term,
namely that for $k=n$) and the weight function to
$$
\lim_{b\to g^2}h_n(g,b,b\,|q)=\frac{1-g^2}{1-g^2q^{2n}}q^n.
$$
This family of orthogonal functions
$\big(\!\lim_{b\to g^2}F_n(z;g,b,b\,|q)\big)_{n\in Z}$ coincides with
the Fourier family $(z^n)_{n\in\Z}$ after renormalization.

Finally, we notice the following representation
for the functions $F_n$ in \eqref{eq:fn} in terms of multiples
of well-poised bilateral basic hypergeometric $_3\psi_3$ series:
\begin{equation}
F_n(z;g,b,c\,|q)=\frac{(b,b/g^2;q)_n}{(cq,cq/g^2;q)_n}
\left(\frac{cq}b\right)^n\,{}_3\psi_3\!
\left[\begin{matrix}gq,bq^n,g^2q^{-n}/c\\
g,g^2q^{1-n}/b,cq^{1+n}\end{matrix};q,q^{\frac 12}z\right].
\end{equation}
It is clear that since $0<q<1$
these functions converge absolutely on the unit circle.
\end{remark}
\begin{proof}[Proof of Theorem~\ref{thm:orthf}]
We have
\begin{align*}
  \langle F_n,F_m\rangle
  &=\frac 1{2\pi\mathrm i}\int_{|z|=1}F_n(z;g,b,c\,|q)
     \overline{F_m(z;g,b,c\,|q)}\,
\frac{{\mathrm d}z}z\\
  &=\frac 1{2\pi\mathrm i}\int_{|z|=1}F_n(z;g,b,c\,|q)
    F_m(z^{-1};-g,c,b\,|q)\,
\frac{{\mathrm d}z}z\\
&=\frac 1{2\pi\mathrm i}\int_{|z|=1}
\sum_{k,l\in\Z}F_{n,k}(g,b,c\,|q)F_{m,l}(-g,c,b\,|q)z^{k-l}\,
\frac{{\mathrm d}z}z
\\&
=\sum_{k,l\in\Z}F_{n,k}(g,b,c\,|q)F_{m,l}(-g,c,b\,|q)\,\delta_{k,l}
=\sum_{k\in\Z}F_{n,k}(g,b,c\,|q)F_{m,k}(-g,c,b\,|q).
\end{align*}
(The penultimate equality used \eqref{eq:Fo} and the interchange
of a the double sum with an integral.
Since the double sum converges absolutely and the integrand is
a continuous function, the interchange of the double sum with the
integral can be justified by appealing to a suitable variant of the
Fubini--Tonelli theorem for interchanging integrals
where one of the integrals involves a discrete measure
and the other a complex measure.)
Now the last sum simplifies
by the $(a,b,c,d,e)\mapsto(g^2,bq^n,cq^m,g^2q^{-n}/c,g^2q^{-m}/b)$
special case of Bailey's ${}_6\psi_6$ summation in \eqref{eq:6psi6}
to $h_n\delta_{n,m}$ since the closed form product contains the factor
\begin{equation*}
(q^{1+n-m},q^{1-n+m};q)_\infty
\end{equation*}
which vanishes unless $n=m$. The computational details for the described
final evaluation are as follows:
\begin{align*}
&\sum_{k\in\Z}F_{n,k}(g,b,c\,|q)F_{m,k}(-g,c,b\,|q)\\
  &=\sum _{k=-\infty}^{\infty}\frac{(1-gq^k)}{(1-g)}
    \frac{(b;q)_{n+k}(g^2/c;q)_{k-n}}{(cq;q)_{n+k}(g^2q/b;q)_{k-n}}q^{\frac k2}
    \frac{(1+gq^k)}{(1+g)}
    \frac{(c;q)_{m+k}(g^2/b;q)_{k-m}}
    {(bq;q)_{m+k}(g^2q/c;q)_{k-m}}q^{\frac k2}\\
  &= \frac{(b;q)_{n}(g^2/c;q)_{-n}}{(cq;q)_{n}(g^2q/b;q)_{-n}}
    \frac{(c;q)_{m}(g^2/b;q)_{-m}}{(bq;q)_{m}(g^2q/c;q)_{-m}}\\
  &\quad\;\times\sum _{k=-\infty}^{\infty}\frac{(1-g^2q^{2k})}{(1-g^2)}
    \frac{(bq^n,g^2q^{-n}/c,cq^m,g^2q^{-m}/b;q)_k}
    {(g^2q^{1-n}/b,cq^{1+n},g^2q^{1-m}/c,bq^{1+m};q)_k}q^k\\
  &= \frac{(b;q)_{n}(g^2/c;q)_{-n}}{(cq;q)_{n}(g^2q/b;q)_{-n}}
    \frac{(c;q)_{m}(g^2/b;q)_{-m}}{(bq;q)_{m}(g^2q/c;q)_{-m}}\\
  &\quad\;\times\frac{(q,g^2q,q/g^2,cq/b,g^2q^{1-n-m}/bc,
    q^{1-n+m},q^{1+n-m},bcq^{1+n+m}/g^2,bq/c;q)_{\infty}}
    {(g^2q^{1-n}/b,cq^{1+n},g^2q^{1-m}/c,bq^{1+m},q^{1-n}/b,
    cq^{1+n}/g^2,q^{1-m}/c,bq^{1+m}/g^2,q;q)_{\infty}}\\
  &= \frac{(q,q,g^2q,q/g^2,g^2q/bc,bcq/g^2,cq/b,bq/c;q)_\infty}
    {(g^2q/b,bq/g^2,g^2q/c,cq/g^2,bq,q/b,cq,q/c;q)_\infty}
    \frac{(1-bc/g^2)}{(1-bcq^{2n}/g^2)}q^n\,\delta_{n,m},
\end{align*}
where we have used some elementary identities
involving $q$-shifted factorials to simplify the products
(cf.\ \cite[Appendix~I]{GR}).
\end{proof}

\begin{remark}
While the family of functions $\big(F_n(z;g,b,c\,|q)\big)_{n\in\Z}$
with $\Re(g)=0$ and $c=\overline{b}$ are orthogonal
on the unit circle, the summation
\begin{equation}\label{eq:Fo1}
\sum_{k\in\Z}F_{n,k}(g,b,c\,|q)F_{m,k}(-g,c,b)=h_n(g,b,c\,|q)\,\delta_{n,m}.
\end{equation}
still holds when $g\in\C$, and $b$ and $c$ are independent
(i.e., the conditions $\Re(g)=0$ and
$c=\overline{b}$ are not necessary), as long as
$|c^2/b^2|<|q^{\frac 12}|(<1)$ hold (which is required
for the absolute convergence of the sum).

We conclude that the two sequences of functions
\begin{equation}\label{eq:orel}
\big(F_n(z;g,b,c\,|q)\big)_{n\in\Z}\quad\ \text{and}\quad\ 
\big(F_m(z;-\overline{g},\overline{c},\overline{b}\,|q)\big)_{m\in\Z},
\end{equation}
are \textit{biorthogonal} to each other
on the unit circle, with respect to the weight function $w(z)$=1
(while relaxing the condition of positive definiteness
of $h_n$ that one usually demands when regarding orthogonality;
in our case we may even choose $q$ to be complex with $0<|q|<1$).
Equivalently (in view of \eqref{eq:Fo1})
we have the \textit{discrete biorthogonality}
of the doubly indexed
sequences (with first indices specifying the \textit{order}
of the respective elements, the second indices being control variables)
\begin{equation*}
\big(F_{n,k}(g,b,c\,|q)\big)_{n,k\in\Z}\quad\ \text{and}\quad\ 
\big(F_{m,l}(-\overline{g},\overline{c},\overline{b}\,|q)\big)_{m,l\in\Z},
\end{equation*}
for $0<q<1$ and $g,b,c\in\C\setminus\{0\}$,
by which we just mean that
$$\sum_{k\in\Z}F_{n,k}(g,b,c\,|q)
\overline{F_{m,k}(-\overline{g},\overline{c},\overline{b}\,|q)}
=h_n(g,b,c)\,\delta_{n,m}$$
holds for some $h_n(g,b,c)$,
which is exactly \eqref{eq:Fo1}.
\end{remark}
\begin{remark}
  In general, given a matrix inverse $G=F^{-1}$ where
$F=(f_{nk})_{n,k\in\mathbb Z}$ and
$G=(g_{nk})_{n,k\in\mathbb Z}$ (explained in more detail in
Section~\ref{sec:mi}),
we can use the Fourier orthogonality \eqref{eq:Fo} to
define two sequences of functions $F_n(z)=\sum_{k\in\mathbb Z}f_{nk}z^k$
and $G_m(z)=\sum_{l\in\mathbb Z}\bar{g_{lm}}z^l$,
$n,m\in\mathbb Z$
which (subject to suitable conditions of convergence)
form a sequence of (bilateral) biorthogonal functions
on the unit circle.  
\end{remark}

\section[Bilateral and multilateral matrix inverses]{Bilateral
  and multilateral matrix inverses}\label{sec:mi}

We now turn to (bilateral) \textit{matrix inverses}.
We consider infinite matrices
$F=(f_{nk})_{n,k\in\mathbb Z}$ and
$G=(g_{nk})_{n,k\in\mathbb Z}$, and infinite sequences
$(a_n)_{n\in\mathbb Z}$ and $(b_n)_{n\in\mathbb Z}$.
Further, we choose the entries of these matrices and sequences
to be such that (all) the sums below converge absolutely.

$F=(f_{nk})_{n,k\in\mathbb Z}$ and $G=(g_{kl})_{k,l\in\mathbb Z}$ are
\textit{inverses} of each other if and only if
\begin{align*}
  \sum _{k\in\mathbb Z}f_{nk}g_{kl}&=\delta_{nl}\qquad\qquad
  \text {for all}\quad n,l\in\mathbb Z\\
  \intertext{and}
  \sum _{k\in\mathbb Z}g_{nk}f_{kl}&=\delta_{nl}\qquad\qquad
\text {for all}\quad n,l\in\mathbb Z
\end{align*}
hold. (As the two matrices $F$ and $G$
are infinite and not necessarily (lower-)triangular,
the two sums above may be infinite and the validity
of one equation need not imply the validity of the other.)
Further,
\begin{subequations}\label{eq:fg}
\begin{align}
  \sum _{n\in\mathbb Z}f_{n k}a_{n}&=b_{k}\qquad\qquad
  \text {for all $k\in\mathbb Z$,}\label{eq:fab}\\
\intertext{if and only if}
  \sum _{k\in\mathbb Z}g_{k l}b_{k}&=a_{l}\qquad\qquad
  \text {for all $l\in\mathbb Z$.}\label{eq:gba}
\end{align}
\end{subequations}

This similarly applies to multivariate sequences involving multisums.
The two equations in \eqref{eq:fg} are usually called \textit{inverse relations}.

In a very similar way as in the proof of Theorem~\ref{thm:orthf},
Bailey's ${}_6\psi_6$ summation \eqref{eq:6psi6}
was used in \cite[Theorem~3.1]{S3} to derive the following
bilateral matrix inverse:

\textit{The infinite matrices
$F=(f_{nk})_{n,k\in\mathbb Z}$ and $G=(g_{kl})_{k,l\in\mathbb Z}$ are
{\em inverses} of each other where
\begin{subequations}\label{eq:qmi}
\begin{align}
f_{nk}&=\frac{(aq/b,bq/a,aq/c,cq/a,bq,q/b,cq,q/c;q)_{\infty}}
{(q,q,aq,q/a,aq/bc,bcq/a,cq/b,bq/c;q)_{\infty}}\notag\\
&\quad\;\times\frac{(1-bcq^{2n}/a)}{(1-bc/a)}\frac{(b;q)_{n+k}\,(a/c;q)_{k-n}}
{(cq;q)_{n+k}\,(aq/b;q)_{k-n}}\\
\intertext{and}
g_{kl}&=\frac{(1-aq^{2k})}{(1-a)}\frac{(c;q)_{k+l}\,(a/b;q)_{k-l}}
{(bq;q)_{k+l}\,(aq/c;q)_{k-l}}\,q^{k-l}.
\end{align}
\end{subequations}}

This result can be used to derive a bilateral summation from a suitable
given summation (where the given summation need not be a bilateral summation
but could also just be a unilateral summation!).
In particular, if we suitably choose $b_k$ and $a_l$ such that
\eqref{eq:gba} holds by some known summation (say, by
Jackson's ${}_8\phi_7$ summation \eqref{eq:8phi7}),
then the inverse relation \eqref{eq:fab}
(subject to absolute convergence of the series)
automatically must be true. Using exactly this method
the following particular well-poised balanced
\textit{${}_8\psi_8$ summation}
was derived in \cite[Theorem~4.1]{S3}:

\textit{Let $a$, $b$, $c$, and $d$ be indeterminates,
let $k$ be an arbitrary integer and $M$ a non-negative integer. Then
\begin{align}\label{eq:8psi8}
&{}_8\psi_8\!\left[\begin{matrix}qa^{\frac 12},-qa^{\frac 12},b,c,dq^k,
aq^{-k}/c,aq^{1+M}/b,aq^{-M}/d\notag\\
a^{\frac 12},-a^{\frac 12},aq/b,aq/c,aq^{1-k}/d,
cq^{1+k},bq^{-M},dq^{1+M}\end{matrix}\,;q,q\right]\\
&=\frac {(aq/bc,cq/b,dq,dq/a;q)_M}
{(cdq/a,dq/c,q/b,aq/b;q)_M}\notag\\
&\quad\;\times
\frac{(cd/a,bd/a,cq,cq/a,dq^{1+M}/b,q^{-M};q)_k}
{(q,cq/b,d/a,d,bcq^{-M}/a,cdq^{1+M}/a;q)_k}\notag\\
&\quad\;\times
\frac{(q,q,aq,q/a,cdq/a,aq/cd,cq/d,dq/c;q)_{\infty}}
{(cq,q/c,dq,q/d,cq/a,aq/c,dq/a,aq/d;q)_{\infty}}.
\end{align}}
In \cite[p.~345]{S7} it was explained how to deduce Jackson's
$_8\phi_7$ summation in \eqref{eq:8phi7} and Bailey's $_6\psi_6$
summation in \eqref{eq:6psi6} from \eqref{eq:8psi8} by suitable
limiting cases and a polynomial argument (resp., analytic continuation).

Likewise, we can utilize multivariate ${}_6\psi_6$ summations associated to
root systems (see \cite[Section~2.4]{S8}) to obtain \textit{multilateral
  matrix inverses}
which in combination with existing multivariate Jackson summations
(see \cite[Section~2.3]{S8}) can be used to deduce
\textit{multivariate well-poised and balanced ${}_8\psi_8$ summations}.
For instance, Gustafson's $A_r$ $_6\psi_6$
summation from \cite[Theorem~1.15]{G} and Milne's $A_r$ $_8\phi_7$
summation from \cite[Theorem~6.17]{M1} served as ingredients
in the derivation of the following $A_r$ extension of \eqref{eq:8psi8},
derived in \cite[Theorem~4.1]{S7}:

\textit{Let $a,b,c_1\ldots,c_j,d,x_1\ldots,x_r$ be indeterminates,
  let $k_1,\ldots,k_r$ be arbitrary integers and $M$ a non-negative integer.
  Then}
\begin{align}\label{eq:ar8psi8}
\sum_{n_1,\dots,n_r=-\infty}^{\infty}&\left(
\prod_{1\le i<j\le r}\frac {x_iq^{n_i}-x_jq^{n_j}}{x_i-x_j}
\prod_{i=1}^r\frac{1-ax_iq^{n_i+|n|}}{1-ax_i}
\prod_{i,j=1}^r\frac{(c_jx_i/x_j;q)_{n_i}}
{(q^{1+k_j}c_jx_i/x_j;q)_{n_i}}\right.\notag\\
&\left.\times
\prod_{i=1}^r\frac{(ax_iq^{-k_i}/c_i;q)_{|n|}\,
(bx_i,ax_iq^{-M}/d;q)_{n_i}}
{(ax_iq/c_i;q)_{|n|}\,
(bx_iq^{-M},ax_iq^{1-|k|}/d;q)_{n_i}}\cdot
\frac{(dq^{|k|},aq^{1+M}/b;q)_{|n|}}
{(dq^{1+M},aq/b;q)_{|n|}}\,q^{|n|}\right)\notag\\
&=\prod_{i,j=1}^r\frac{(qc_jx_i/c_ix_j,qx_i/x_j;q)_{\infty}}
{(qc_jx_i/x_j,qx_i/c_ix_j;q)_{\infty}}
\prod_{i=1}^r\frac{(ax_iq,q/ax_i,ax_iq/c_id,c_idq/ax_i;q)_{\infty}}
{(ax_iq/c_i,c_iq/ax_i,ax_iq/d,dq/ax_i;q)_{\infty}}\notag\\
&\quad\;\times
\frac{(dq/C,Cq/d;q)_{\infty}}{(dq,q/d;q)_{\infty}}\,
\frac{(dq,aq/bC;q)_M}{(aq/b,dq/C;q)_M}\,
\prod_{i=1}^r\frac{(c_iq/bx_i,dq/ax_i;q)_M}{(c_idq/ax_i,q/bx_i;q)_M}\notag\\
&\quad\;\times
\frac{(bd/a,q^{-M};q)_{|k|}}
{(d,bCq^{-M}/a;q)_{|k|}}
 \prod_{i,j=1}^r\frac{(qc_jx_i/x_j;q)_{k_i}}
{(qc_jx_i/c_ix_j;q)_{k_i}}\notag\\
&\quad\;\times
\prod_{i=1}^r\frac{(c_id/ax_i;q)_{|k|}\,
(c_iq/ax_i,c_idq^{1+M}/bCx_i;q)_{k_i}}
{(d/ax_i;q)_{|k|}\,(c_iq/bx_i,c_idq^{1+M}/ax_i;q)_{k_i}}.
\end{align}
Just as in the $r=1$ case, this result can be used to deduce
(the two aforementioned ingredients, namely) Milne's $A_r$ $_8\phi_7$
summation and Gustafson's $A_r$ $_6\psi_6$ summation by suitable
limiting cases and a polynomial argument
(resp., repeated application of analytic continuation),
see \cite[Remark~4.2]{S7}.
The paper \cite{S7} contains, in addition to \eqref{eq:ar8psi8},
two other multivariate $_8\psi_8$ summations that extend the summation
\eqref{eq:8psi8} to multi-sums.

Given this line of results, one may wonder if one might be able to
similarly obtain results for ordinary (not basic) hypergeometric
series by using Dougall's $_2H_2$ summation formula in \eqref{eq:2H2}
as a starting point (instead of Bailey's $_6\psi_6$ summation
formula in \eqref{eq:6psi6}) in combination with other results.
This is indeed the case.
We have the following new result:
\begin{theorem}[A bilateral matrix inverse]\label{thm:bmi}
  Let $a,c\in\C\setminus\Z$.
  The infinite matrices 
$F=(f_{nk})_{n,k\in\mathbb Z}$ and $G=(g_{kl})_{k,l\in\mathbb Z}$ are
inverses of each other where
\begin{subequations}
\begin{equation}\label{eq:bmfnk}
f_{nk}=\frac{\Ga(1+a-c)}
{\Ga(1+a+n+k,1-c-n-k)}
\end{equation}
and
\begin{equation}
g_{kl}=\frac{\Ga(1+c-a)}
{\Ga(1+c+l+k,1-a-l-k)}.
\end{equation}
\end{subequations}
\end{theorem}
The matrix inverse in Theorem~\ref{thm:bmi} (which does \textit{not}
appear to follow from \eqref{eq:qmi} by a direct limit)
can be recast as follows:

\begin{theoremp}{thm:bmi}[A bilateral matrix inverse]\label{thm:bmip}
Let $a,c\in\C\setminus\Z$ and the entries of the infinite matrix
  $F=(f_{nk})_{n,k\in\mathbb Z}=(f_{nk}(a,c\,|q))_{n,k\in\mathbb Z}$
  be given as in \eqref{eq:bmfnk}. Then the inverse of $F$ is the matrix
  $F^{-1}=(f_{nk}(c,a\,|q))_{n,k\in\mathbb Z}$.
\end{theoremp}
\begin{remark}
  In this remark, we shall write
  $f_{n,k}$ and $g_{k,l}$, instead of $f_{nk}$ and $g_{kl}$,
respectively, in order to better distinguish the two indices.
Since
 $$
\sum_{k\in\mathbb Z}f_{n,k}g_{k,l}=\delta_{n,l}=\delta_{-n,-l}=
\sum_{k\in\mathbb Z}f_{-n,k}g_{k,-l},
$$
it is clear that Theorem~\ref{thm:bmi} is equivalent to the
following assertion (where $\tilde{f}_{nk}=f_{-n,k}$ and
$\tilde{g}_{kl}=g_{k,-l}$):

\textit{Let $a,c\in\mathbb C\setminus\mathbb Z$.
  The infinite matrices 
  $\tilde{F}=(\tilde{f}_{nk})_{n,k\in\mathbb Z}$ and
  $\tilde{G}=(\tilde{g}_{kl})_{k,l\in\mathbb Z}$ are
inverses of each other where}
\begin{subequations}\label{eq:bmfnkv1}
\begin{equation}
\tilde{f}_{nk}=\frac{\Gamma(1+a)}
{\Gamma(1+a-n+k,1-c+n-k)}
\end{equation}
 \textit{and}
\begin{equation}
\tilde{g}_{kl}=\frac{\Gamma(1-a)}
{\Gamma(1+c-l+k,1-a+l-k)}.
\end{equation} 
\end{subequations}
Letting $c\to 0$ in this matrix inverse we obtain
\begin{subequations}\label{eq:bmfnkv2}
\begin{equation}
\tilde{f}_{nk}=(-1)^{n-k}\frac{(-a)_{n-k}}
{(n-k)!}
\end{equation}
and
\begin{equation}
\tilde{g}_{kl}=(-1)^{k-l}\frac{(a)_{k-l}}
{(k-l)!}.
\end{equation}
\end{subequations}
This is a matrix inverse involving lower-triangular matrices
that one can easily deduce
from the classical Vandermonde convolution formula
$$
\binom{a+b}{n}=\sum_{k=0}^n\binom ak\binom b{n-k}
$$
(which can be seen to be a special case of Dougall's $_2H_2$ summation
formula used to prove Theorem~\ref{thm:bmi}).
The matrix inverse in \eqref{eq:bmfnkv2}
is extracted from this summation,
or rather the orthogonality relation obtained from it,
by replacing $n$ by $n-l$ and letting $b\to -a$
(the product side then becoming $\delta_{n,l}$).
This means that the matrix inverse in
Theorem~\ref{thm:bmi} can be viewed as a bilateral
extension of the matrix inverse in \eqref{eq:bmfnkv2}.
A similar remark applies to our multivariate extension of
Theorem~\ref{thm:bmi} in Theorem~\ref{thm:mmi} which can be viewed
as an extension of a matrix inversion result that holds for
a pair of lower-triangular matrices (that are
indexed by a pair of multi-indices).
\end{remark}

\begin{remark}
A referee has enquired about the relation of the matrix inverse
stated in Theorem~\ref{thm:bmi} and Bressoud's matrix inverse \cite{Br},
and its specializations such as Andrews' matrix inverse (cf.\ \cite{A2}).

The Bressoud matrix inverse can be regarded as a ``well-poised''
matrix inverse and can be stated as follows:
\textit{Let $a,b\in\mathbb C$ such that the following expressions
  have no poles. Then the infinite matrices 
$F=(f_{nk})_{n,k\in\mathbb Z}$ and $G=(g_{kl})_{k,l\in\mathbb Z}$ are
\textit{inverses} of each other where}
\begin{subequations}\label{eq:bmfnkv3}
\begin{equation}
f_{nk}=\frac{1-bq^{2n}}{1-bq^{2k}}\frac{(b;q)_{n+k}(b/a;q)_{n-k}}
{(aq;q)_{n+k}(q;q)_{n-k}}a^{n-k}
\end{equation}
and
\begin{equation}
g_{kl}=\frac{1-aq^{2k}}{1-aq^{2l}}\frac{(a;q)_{k+l}(a/b;q)_{k-l}}
{(bq;q)_{k+l}(q;q)_{k-l}}b^{k-l}.
\end{equation}
\end{subequations}
Bressoud's matrix inverse \eqref{eq:bmfnkv3}
exhibits an $a\leftrightarrow b$ symmetry as observed by Bressoud
in his paper. Bressoud's result (which has been considerably further
generalized by Krattenthaler~\cite{Kr} using an operator method)
can be directly extracted from the
terminating very-well-poised $_6\phi_5$ summation
(cf.\ \cite[Equation (II.21)]{GR}). A relevent special case of
\eqref{eq:bmfnkv3} is Andrew's matrix inverse (cf.\ \cite{A2})
used in the Bailey transform. That matrix inverse can
be obtained by taking the limit $b\to 0$ in Bressoud's matrix inverse,
or by directly extracting the result from the terminating
very-well-poised $_4\phi_3$ summation
\cite[$c\to aq/b$ in Equation (II.21)]{GR}.
However, there is also another important limit of Bressoud's matrix inverse,
obtained by replacing $a$ and $b$ by $at$ and $bt$, respectively,
and subsequently letting $t\to 0$. This leads to a simple
matrix inverse that can be directly extracted from
the terminating $q$-Chu--Vandermonde summation
(cf.\ \cite[Equation (II.6)]{GR}).
A suitable $q\to 1$ limit of that matrix inverse is exactly
the result that is ``bilateralized'' in Theorem 2.
\end{remark}

In Theorem~\ref{thm:mmi}, which is our multivariate extension
of Theorem~\ref{thm:bmi},
we do not have the ($a\leftrightarrow c$) symmetry that prevails
in Theorem~\ref{thm:bmip}.
For this reason we decided here to use Theorem~\ref{thm:bmi}
(that uses two matrices $F$ and $G$) as the primary formulation
of our bilateral inversion result, instead of the more elegant
formulation in Theorem~\ref{thm:bmip}.
\begin{proof}[Proof of Theorem~\ref{thm:bmi}]
  It is enough to show the relation $\sum_{k\in\Z}f_{nk}g_{kl}=\delta_{n,l}$
  (as the dual relation $\sum_{k\in\Z}g_{nk}f_{kl}=\delta_{n,l}$ is
  just the same with $a$ and $c$ interchanged). We have
\begin{align*}
  \sum_{k\in\Z}f_{nk}g_{kl}
  &=\sum_{k=-\infty}^\infty\frac{\Ga(1+a-c)}
{\Ga(1+a+n+k,1-c-n-k)}\frac{\Ga(1+c-a)}{\Ga(1+c+l+k,1-a-l-k)}\\
  &=\frac{\Ga(1+a-c,1+c-a)}{\Ga(1+a+n,1-c-n,1+c+l,1-a-l)}
    \sum_{k=-\infty}^\infty\frac{(a+l)_k(c+n)_k}{(1+a+n)_k(1+c+l)_k}\\
  &=\frac{\Ga(1+a-c,1+c-a)}{\Ga(1+a+n,1-c-n,1+c+l,1-a-l)}\\
  &\quad\;\times\frac{\Ga(1-a-l,1-c-n,1+c+l,1+a+n,1)}
    {\Ga(1+c-a,1+l-n,1+n-l,1+a-c)}\\
  &=\frac 1{\Ga(1+l-n,1+n-l)}=\delta_{n,l},
\end{align*}
where we have used the $(a,b,c,d)\mapsto(a+l,c+n,1+c+l,1+a+n)$ case
of Dougall's $_2H_2$ summation formula in \eqref{eq:2H2}
(which converges absolutely for the specified choice of parameters),
and the fact that the product of gamma functions
$\Ga(1+l-n,1+n-l)^{-1}$ vanishes unless $n=l$.
\end{proof}

Now, if we suitably choose $b_k$ and $a_l$ such that \eqref{eq:gba}
holds by, say, the Pfaff--Saalsch\"utz ${}_3F_2$ summation in \eqref{eq:3F2},
then the inverse relation \eqref{eq:fab}
automatically must be true. This procedure gives
us the following $2$-balanced \textit{${}_3H_3$ summation}:
\begin{proposition}[An $_3H_3$ summation]\label{prop:3H3}
  Let $a,b,c$ be indeterminates, let $k$ be an arbitrary integer and $M$
  be a non-negative integer. Then
\begin{align}\label{eq:3H3}
{}_3H_3\!\left[\begin{matrix}a,\,c+k,\,1+c-b+M\\
1+a+k,\,1+c+M,\,1+c-b\end{matrix}\,;1\right]=
\frac{\Ga(1+a,1-a,1+c,1-c)}
{\Ga(1+a-c,1+c-a)}&\notag\\
\times\frac{(-M,1+a,b)_k}{(1,c,a+b-c-M)_k}
\frac{(1+c,1+c-a-b)_M}{(1+c-a,1+c-b)_M}&.
\end{align}
\end{proposition}
\begin{proof}
We apply Theorem~\ref{thm:bmi} in conjunction with the equivalence
of the two relations \eqref{eq:fab} and \eqref{eq:gba}.
Choosing
$$
a_l=\frac{\Ga(1+c-a)}{\Ga(1-a,1+c)}\frac{(1+c-a)_M}{(1+c)_M}
\frac{(a,1+c-b+M)_l\,(-1)^l}{(1+c-b,1+c+M)_l}
$$
and
$$
b_k=\frac{(b,-M)_k\,(-1)^k}{(1,a+b-c-M)_k},
$$
we see that the relation \eqref{eq:gba} holds by the
$(a,b,c,m)\mapsto(a,b,1+c+l,M)$ case of the
Pfaff--Saalsch\"utz ${}_3F_2$ summation in \eqref{eq:3F2}.
Thus the inverse relation \eqref{eq:fab} (which converges absolutely)
must be true, which after some rewriting is \eqref{eq:3H3}.
\end{proof}
\begin{remark}\label{rem:2l}
Two special cases of Proposition~\ref{prop:3H3} are worth pointing out:
\begin{enumerate}
\item
  If $a\to -k$ ($k$ being a non-negative integer),
  then the bilateral series in \eqref{eq:3H3}
gets truncated from below and from above so that the sum is finite. By
a polynomial argument, $M$ can then be replaced by any complex number.
If we replace $M$ by $A+b-c-1$, then perform the simultaneous substitution
$b\mapsto C-B$ and $c\mapsto A-k$ we obtain the
Pfaff--Saalsch\"utz summation in \eqref{eq:3F2}
(subject to the substitution $(a,b,c,m)\mapsto(A,B,C,k)$).

\item
  If, in \eqref{eq:3H3}, we formally let $M\to\infty$
  (which can be justified by appealing to Tannery's theorem),
  and rewrite the products on the right-hand side which are of the
  form $(x)_k$ as $\Ga(x+k)/\Ga(x)$, we can apply analytic continuation
to replace $k$ by $B-c$ (in order to relax the integrality condition
of $k$) where $B$ is a new complex parameter. We then obtain
an identity with four free parameters that can be seen to be
equivalent to Dougall's ${}_2H_2$ summation in \eqref{eq:2H2}
after substitution of parameters.
\end{enumerate}
\end{remark}

For the derivation of our multivariate extensions of
Theorems~\ref{thm:bmi} and Proposition~\ref{prop:3H3} we shall
make use of the following \textit{$A_r$ Pfaff--Saalsch\"utz summation}
(which can be obtained as a suitable $q\to 1$ limit of a result by
Milne \cite[Theorem~4.15]{M2}).
\begin{align}\label{eq:ar3F2}
\sum_{\substack{k_1,\dots,k_r\ge0\\|k|\le M}}\Bigg(
\prod_{1\le i<j\le r}\frac{x_i+k_i-x_j-k_j}{x_i-x_j}
\prod_{i,j=1}^r\frac{(a_j+x_i-x_j)_{k_i}}{(1+x_i-x_j)_{k_i}}
\prod_{i=1}^r\frac{(b+x_i)_{k_i}}{(c+x_i)_{k_i}}&\notag\\
\times
\frac{(-M)_{|k|}}{(a_1+\cdots+a_r+b+1-M-c)_{|k|}}&\Bigg)\notag\\
=\frac{(c-b)_M}{(c-a_1-\cdots-a_r-b)_M}
\prod_{i=1}^r\frac{(c+x_i-a_i)_M}{(c+x_i)_M}&.
\end{align}

Further, we shall make use of Gustafson's
\textit{$A_r$ extension of Dougall's ${}_2H_2$ summation}
\cite[Theorem~1.11]{G}.
\begin{align}\label{eq:ar2H2}
&\sum_{k_1,\dots,k_r=-\infty}^\infty
\prod_{1\le i<j\le r}\frac{x_i+k_i-x_j-k_j}{x_i-x_j}
\prod_{i=1}^r\prod_{j=1}^{r+1}\frac{(a_j+x_i)_{k_i}}{(b_j+x_i)_{k_i}}\notag\\
&=\frac{\Gamma\big({-}r+\sum_{j=1}^{r+1}(b_j-a_j)\big)
\prod_{i=1}^r\prod_{j=1}^{r+1}\Gamma(1-a_j-x_i,b_j+x_i)}
{\prod_{i,j=1}^{r+1}\Gamma(b_j-a_i)
\prod_{i,j=1}^{r}\Gamma(1+x_j-x_i)},
\end{align}
provided $\Re\big({-}r+\sum_{j=1}^{r+1}(b_j-a_j)\big)>0$.

A special case of \eqref{eq:ar2H2}
is the following \textit{$A_r$ ${}_1H_1$ summation}
that extends \eqref{eq:1H1}.
\begin{align}
\sum_{k_1,\dots,k_r=-\infty}^\infty
\prod_{1\le i<j\le r}\frac{x_i+k_i-x_j-k_j}{x_i-x_j}
\prod_{i,j=1}^r\frac{(a_j+x_i)_{k_i}}{(b_j+x_i)_{k_i}}
\cdot z^{|k|}&\notag\\
=\frac{(1-z)^{-r+\sum_{j=1}^r(b_j-a_j)}}{(-z)^{{-}r+\sum_{j=1}^r(b_j+x_j)}}
\prod_{i,j=1}^r\frac{\Gamma(1-a_j-x_i,b_j+x_i)}
{\Gamma(b_j-a_i,1+x_j-x_i)}&,
\end{align}
provided $|z|=1$, $z\ne 1$, and $\Re\big({-}r+\sum_{j=1}^{r}(b_j-a_j)\big)>0$.

As a consequence of Gustafson's $A_r$ ${}_2H_2$ summation the following holds:

\begin{theorem}[An $A_r$ multilateral matrix inverse]\label{thm:mmi}
Let $a_1,\ldots,a_r,c,x_1,\ldots,x_r$ be indeterminates. Then the infinite matrices
$F=(f_{nk})_{n,k\in\mathbb Z^r}$ and $G=(g_{kl})_{k,l\in\mathbb Z^r}$ are
\textit{inverses} of each other where
\begin{equation*}
f_{nk}=
\prod_{i=1}^r\frac{\Ga(1+a_i+n_i-|n|-c)}{\Ga(1-c-|n|-x_i-k_i)}
\prod_{i,j=1}^r\frac1{\Ga(1+a_j+n_j+x_i+k_i)}
\end{equation*}
and
\begin{align*}
g_{kl}={}&\prod_{1\le i<j\le r}\frac{x_i+k_i-x_j-k_j}{x_i-x_j}
\prod_{i=1}^r\frac{\Ga(1+c+|l|-l_i-a_i)}
{\Ga(1+c+|l|+x_i+k_i)}\\
&\times(-1)^{(r-1)|k|}\prod_{i,j=1}^r\frac{\Ga(1+x_j-x_i,1+a_j-a_i+l_j-l_i)}
{\Ga(1-a_j-l_j-x_i-k_i)}.
\end{align*}
\end{theorem}
\begin{proof}
We will show the relation $\sum_{k\in\Z^r}f_{nk}g_{kl}=\delta_{n,l}$.
(The verification of the dual relation $\sum_{k\in\Z^r}g_{nk}f_{kl}=\delta_{n,l}$
can be done analogously.) We have
\begin{align*}
  &\sum_{k\in\Z^r}f_{nk}g_{kl}
  =\sum_{k_1,\ldots,k_r=-\infty}^\infty
  \left(\prod_{i=1}^r\frac{\Ga(1+a_i+n_i-|n|-c)}{\Ga(1-c-|n|-x_i-k_i)}
\prod_{i,j=1}^r\frac1{\Ga(1+a_j+n_j+x_i+k_i)}\right.\\
&\qquad\qquad\qquad\qquad\qquad\quad\times
\prod_{1\le i<j\le r}\frac{x_i+k_i-x_j-k_j}{x_i-x_j}
\prod_{i=1}^r\frac{\Ga(1+c+|l|-l_i-a_i)}
{\Ga(1+c+|l|+x_i+k_i)}\\
&\left.\qquad\qquad\qquad\qquad\qquad\quad\times
(-1)^{(r-1)|k|}\prod_{i,j=1}^r\frac{\Ga(1+x_j-x_i,1+a_j-a_i+l_j-l_i)}
{\Ga(1-a_j-l_j-x_i-k_i)}\right)\\
&=\prod_{i,j=1}^r\frac{\Ga(1+x_j-x_i,1+a_j-a_i+l_j-l_i)}
{\Ga(1-a_j-l_j-x_i,1+a_j+n_j+x_i)} \\
  &\quad\;\times
    \prod_{i=1}^r
\frac{\Ga(1+a_i+n_i-|n|-c,1+c+|l|-l_i-a_i)}
{\Ga(1-c-|n|-x_i,1+c+|l|+x_i)}\\
&\times
\sum_{k_1,\ldots,k_r=-\infty}^\infty
\prod_{1\le i<j\le r}\!\!\frac{x_i+k_i-x_j-k_j}{x_i-x_j}
  \prod_{i,j=1}^r\frac{(a_j+l_j+x_i)_{k_i}}{(1+a_j+n_j+x_i)_{k_i}}
 \prod_{i=1}^r\frac{(c+|n|+x_i)_{k_i}}{(1+c+|l|+x_i)_{k_i}}\\
 &=\prod_{i,j=1}^r\frac{\Ga(1+x_j-x_i,1+a_j-a_i+l_j-l_i)}
{\Ga(1-a_j-l_j-x_i,1+a_j+n_j+x_i)} \\
&\quad\;\times\prod_{i=1}^r
\frac{\Ga(1+a_i+n_i-|n|-c,1+c+|l|-l_i-a_i)}
{\Ga(1-c-|n|-x_i,1+c+|l|+x_i)}\\
&\quad\;\times\frac 1{\Ga(1+|l|-|n|)}
\prod_{i,j=1}^r\frac{\Ga(1-a_j-l_j-x_i,1+a_j+n_j+x_i)}
{\Ga(1+x_j-x_i,1+a_j-a_i+n_j-l_i)}\\
&\quad\;\times
\prod_{i=1}^r\frac{\Ga(1-c-|n|-x_i,1+c+|l|+x_i)}
{\Ga(1+a_i+n_i-|n|-c,1+c+|l|-l_i-a_i)}\\
&=\frac 1{\Ga(1+|l|-|n|)}
\prod_{i,j=1}^r\frac{\Ga(1+a_j-a_i+l_j-l_i)}
{\Ga(1+a_j-a_i+n_j-l_i)}=\delta_{n,l},
\end{align*}
where we have used Gustafson's $A_r$ $_2H_2$ summation
formula in \eqref{eq:ar2H2} (which converges absolutely for the
specified choice of parameters) with respect to the
following simultaneous substitutions:
$a_j\mapsto a_j+l_j$ ($1\le j\le r$), $a_{r+1}=c+|n|$,
$b_j\mapsto 1+a_j+n_j$ ($1\le j\le r$), and $b_{r+1}=1+c+|l|$,
and further used the fact that the product
$\Ga(1+|l|-|n|)^{-1}\prod_{i,j=1}^r
\Ga(1+a_j-a_i+n_j-l_i)^{-1}$ vanishes unless $n=l$.
\end{proof}

\begin{theorem}[An $A_r$ ${}_3H_3$ summation]\label{thm:ar3H3}
Let $a_1,\ldots,a_r,b,c,x_1,\ldots,x_r$ be indeterminates,
$k_1,\ldots,k_r$ be arbitrary integers and $M$ a non-negative integer.
Then
\begin{align}\label{eq:ar3H3}
\sum_{n\in\Z^r}\prod_{1\le i< j\le r}\frac{a_i+n_i-a_j-n_j}{a_i-a_j}
\prod_{i=1}^r\frac{(1+c-a_i+M)_{|n|-n_i}(c+x_i+k_i)_{|n|}}
{(c-a_i)_{|n|-n_i}(1+c+x_i+M)_{|n|}}&\notag\\\times
\prod_{i,j=1}^r\frac{(a_j+x_i)_{n_j}}{(1+a_j+x_i+k_i)_{n_j}}\cdot
\frac{(1+c-b+M)_{|n|}}{(1+c-b)_{|n|}}&\notag\\
=\prod_{i,j=1}^r\frac{\Ga(1+a_j+x_i,1-a_j-x_i)}{\Ga(1+x_i-x_j,1+a_j-a_i)}
\prod_{i=1}^r\frac{\Ga(1+c+x_i,1-c-x_i)}{\Ga(1+c-a_i,1+a_i-c)}&\notag\\
\times\prod_{i,j=1}^r\frac{(1+a_j+x_i)_{k_i}}{(1+x_i-x_j)_{k_i}}
\prod_{i=1}^r\frac{(b+x_i)_{k_i}}{(c+x_i)_{k_i}}\cdot\frac{(-M)_{|k|}}
{\big(b-c-M+\sum_{j=1}^r(a_j+x_j)\big)_{|k|}}&\notag\\
\times\prod_{i=1}^r\frac{(1+c+x_i)_M}{(1+c-a_i)_M}\cdot
\frac{\big(1+c-b-\sum_{j=1}^r(a_j+x_j)\big)_M}{(1+c-b)_M}.
\end{align}
\end{theorem}
\begin{proof}
We apply Theorem~\ref{thm:mmi} and suitably choose $a_l$ and $b_k$,
with $k,l\in\Z^r$, such that
\begin{equation*}
\sum _{k\in\mathbb Z^r}g_{k l}b_{k}=a_{l}
\end{equation*}
holds by 
the $A_r$ Pfaff--Saalsch\"utz summation in \eqref{eq:ar3F2}.
Once this is achieved, the inverse relation
\begin{equation*}
\sum _{n\in\mathbb Z^r}f_{n k}a_{n}=b_{k}
\end{equation*}
(subject to absolute convergence, which in our case is satisfied)
automatically must be true, yielding, after some rewriting,
the claimed result.
The choice for $a_l$ and $b_k$ that works is
\begin{align*}
a_l&=\prod_{i,j=1}^r\frac{\Ga(1+x_j-x_i,1+a_j-a_i+l_j-l_i)}
{\Ga(1-a_j-l_j-x_i)}
\prod_{i=1}^r\frac{\Ga(1+c-a_i+|l|-l_i)}{\Ga(1+c+|l|+x_i)}\\
&\quad\;\times\frac{(1+c-b+|l|)_M}{\big(1+c-b-\sum_{j=1}^r(a_j+x_j)\big)_M}
\prod_{i=1}^r\frac{(1+c-a_i+|l|-l_i)_M}{(1+c+|l|+x_i)_M}
\end{align*}
and
$$
b_k=\frac{\prod_{i=1}^r(b+x_i)_{k_i}}{\prod_{i,j=1}^r(1+x_i-x_j)_{k_i}}
\cdot\frac{(b)_{|k|}\,(-1)^{|k|}}{\big(b-c-M-\sum_{j=1}^r(a_j+x_j)\big)_{|k|}},
$$
which is a matter of routine to verify.
\end{proof}
\begin{remark}
A similar analysis as described in Remark~\ref{rem:2l} can be applied to
\eqref{eq:ar3H3} in order to recover Milne's $A_r$
Pfaff--Saalsch\"utz summation in \eqref{eq:ar3F2}
or Gustafson's $A_r$ $_2H_2$ summation in \eqref{eq:ar2H2}.
The two relevant limiting cases are as follows:
\begin{enumerate}
\item
  If $a_j\to -x_j-k_j$ ($k_j$ being non-negative integers) for $1\le j\le r$,
then the multilateral series in \eqref{eq:ar3H3}
gets truncated from below and from above so that the sum is finite. By
a polynomial argument, $M$ can then be replaced by any complex number.
We then obtain an identity that can be seen to be equivalent to Milne's
Pfaff--Saalsch\"utz summation in \eqref{eq:ar3F2} after
substitution of parameters.

\item
  If, in \eqref{eq:ar3H3}, we formally let $M\to\infty$
  (which can be justified by appealing to Tannery's theorem),
  and rewrite the products on the right-hand side which are of the
  form $(x)_k$ as $\Ga(x+k)/\Ga(x)$, we can apply analytic continuation
to replace $k_1,\ldots,k_r$ by new complex variables. We then obtain
an identity that can be seen to be equivalent to Gustafson's
$A_r$ ${}_2H_2$ summation in \eqref{eq:ar2H2} after
substitution of parameters.
\end{enumerate}

\end{remark}

\section{Conclusion}\label{sec:con}
Utilizing summations for bilateral and multilateral (basic)
hypergeometric series, we were able to derive two intimately related
but slightly different kinds of results: Firstly, we established
the orthogonality of a specific family of functions on the unit circle
that can be represented in terms of multiples of specific well-poised
bilateral basic hypergeometric $_3\psi_3$ series.
We mention here that so far we failed to extend that
family of orthogonal functions on the unit circle
(i.e., on the $1$-torus) to a (non-trivial) family of multi-indexed
orthogonal functions on the $r$-torus; we would like to propose this
as an open problem.
Secondly, using a similar method, we were able to
obtain bilateral and multilateral matrix inverses,
applicable to bilateral (and multilateral) hypergeometric series.
A notable application of these inverses includes a summation for
a specific multilateral hypergeometric $_3H_3$ series associated
to the root system $A_r$.

\section*{Acknowledgements}
The Author would like to thank the Referees for their
constructive comments which helped to improve the paper.

\section*{Statements and Declarations}
Data sharing is not applicable to this article as no datasets were
generated or analyzed during the current study.
The Author further declares to have no relevant financial
or non-financial interests to disclose.

\end{document}